\documentclass[12pt]{amsart}

\usepackage{amstext,amssymb,amsmath,stmaryrd,amsthm}
\usepackage{geometry} 
\usepackage{color}
\usepackage{hyperref}
\makeindex

\numberwithin{equation}{section}

\newtheorem{theorem}{Theorem}[section]
\newtheorem{cor}[theorem]{Corollary}

\newtheorem{example}{Examples}[section]

\usepackage{amssymb}
\usepackage{amsmath}
\usepackage{color}
\usepackage{graphicx}
\usepackage{comment}
\usepackage{float}
\usepackage{hyperref}

\def\N{{\mathbb N}}

\def\B{{\mathcal B}}

\def\I{\textbf{I}}

\newcommand{\R}{\mathbb{R}}

\newcommand{\E}{\mathbb{E}}

\newcommand{\Ge}{\mathcal{L}}

\begin{document}

\title[Interacting systems of neurons beyond   uniform summability]{Interacting systems of infinite spiking neurons with weights beyond uniform summability}

\author{ Ioannis Papageorgiou }

 \thanks{\textit{Address:}  Universidade Federal do ABC (UFABC) - CMCC, Avenida dos Estados, 5001 - Santo Andre - Sao Paulo, Brasil.
\\ \text{\  \   \      } 
\textit{Email:}  i.papageorgiou@ufabc.edu.br, papyannis@yahoo.com  }




\begin{abstract} 
We consider an infinite system of spiking neurons with a drift and both excitatory and inhibitory connections. We study conditions for non-explosiveness and the uniqueness of the invariant measure. In particular, we examine conditions that allow this infinite interacting system to go beyond the usual interactions of uniformly summable weights. As a result, we extend the Galves-L\"ocherbach model beyond the restrictive uniform summability of the model.
\end{abstract}
 
\date{}
\maketitle 
 
\section{Introduction}

 We study networks containing infinitely many interacting neurons. The activity of each neuron is described by the evolution of its membrane potential (see \cite{Tuc} for a thorough analysis). This evolution occasionally includes a brief depolarization known as a spike.
Spikes, also referred to as jumps (see \cite{H-K-L}, \cite{Kr}), represent a neuron's emission of an action potential (see \cite{L-M}). Each neuron, denoted as $i$, spikes at a rate $\phi_i(x)$, which depends on its membrane potential value $x$. The functions $\phi_i:\R_+ \rightarrow \R_+$ are called intensity functions. At the time of a spike, the neuron's potential is reset to a resting value, which is set to zero in this article. Simultaneously, neurons affected by this spike receive an additional amount of potential (either positive or negative), which is added to their membrane potential. These spikes are the only perturbations of the membrane potential that can be transmitted from one neuron to another through chemical synapses. We investigate systems of interacting particles with variable-length memory, which are systems that depend on the entire history of the system. Such systems have various applications, particularly in describing biological neural networks.

In this paper, we focus on the relationship between inhibitory and excitatory neurons. The balance between inhibitory and excitatory neurons is crucial, as there are examples of neurological networks that exhibit pathological behavior when the number of inhibitory neurons decreases.

One way to examine neural networks is by studying the behavior of individual neurons. Another approach is to analyze the interactions among neurons in the network. In the former case, Hawkes processes are typically used to describe the dynamics characterized by jump times, as seen in \cite{C17}, \cite{D-L-O}, \cite{D-O}, \cite{G-L}, \cite{H-L}, and \cite{H-R-R}. Since neurons reset to zero after spiking, these point processes describing these systems lose their memory.

In the second case, which also applies to the current work, it is possible to describe the activity of the network by modeling the entire evolution of the membrane potential of each neuron rather than just the jump times. In fact, the evolution of the network between jumps is described. In \cite{K-M-R}, \cite{Tu}, and \cite{Cot92}, inhibition models (based on \cite{Kr}) for a finite number of neurons are studied. Since the membrane potential of connected neurons is reduced each time a neuron spikes, the effect of drift is not as crucial in these models as in cases involving excitatory connections. In this paper, we aim to examine different combinations of inhibitory/excitatory connections and the necessary drift to prevent the model from exploding and to ensure the existence of at most one invariant measure in a network of infinitely many neurons. The dynamics of a finite neural network with inhibitory and excitatory connections are studied in \cite{TUROVA1997197}, while networks with only inhibitory or excitatory connections are examined in \cite{Me-Mo-Tu94}. In \cite{H-K-L}, the process describing this evolution follows a deterministic drift between the jumps. In this paper, we will consider both a pure jump process and one with a drift. The process is Markovian and belongs to the family of Piecewise Deterministic Markov Processes introduced by Davis in \cite{Davis84} and \cite{Davis93}. Such processes are widely used in probability modeling of biological or chemical phenomena (see \cite{C-D-M-R}, \cite{PTW-10}, and \cite{ABGKZ}). Furthermore, in \cite{C-D-L-O}, \cite{D-L}, \cite{Co}, \cite{Co2}, and \cite{D-R-O}, mean-field and hydrodynamic limits of the model were studied. In \cite{L17}, regularity properties of the invariant measure were investigated, while in \cite{H-P}, some inequalities were obtained as well. More recently, phase transition and metastability were studied in \cite{F-G-L}, \cite{A}, \cite{A1}, and \cite{A-P}. In \cite{Kr} and \cite{Lo}, no deterministic interactions have been considered.

The neural networks studied in this work consist of an infinite number of neurons that interact with each other through chemical synapses. A neuron's activity is explained by the evolution of its membrane potential. The neurons interact with each other when one neuron spikes, which can be described as a height-amplitude depolarization of its membrane potential. When a neuron, say $j$, spikes, any neuron $i$ connected to it receives an additional amount of membrane potential $W_{ji}\in \R$, called the synaptic weight. For each $j\neq i$, the synaptic weight $W_{ji}$ describes the sole influence of neuron $j$ on neuron $i$. As for the neuron that spikes, its membrane potential is reset to $0$. This description corresponds to the so-called Galves-L\"ocherbach model (G-L) introduced in \cite{G-L} to describe the interaction of brain neural networks. For the existence of the model in \cite{G-L}, the authors require that the synaptic weights satisfy the Property of Uniform Summability (PUS):

\begin{align}\label{Dob}\sup_{i\in \N} \sum_{j\in \N} \vert W_{ji} \vert<\infty.\end{align}

The purpose of this paper is to present conditions that will allow the existence of the model beyond this condition. Since the PUS involves the absolute values of the weights, it does not take into account the influence of negative weights that reduce the membrane potential of the interacting neurons. In this way, a system of inhibitory neurons that, by construction, can interact freely with other neurons is treated in \cite{G-L} as a system of excitatory neurons. Going beyond the uniform summability of (\ref{Dob}) allows us to study models in which a neuron can interact with an arbitrary number of other neurons with weights that do not decrease in absolute value.

The importance of considering both inhibitory and excitatory interactions in a manner where the positive and negative contributions to a neuron cancel each other comes from experimental observations on neural networks. Interruptions of patterned electrical signaling oscillations can be observed in various brain diseases such as epilepsy or Parkinson's disease (see \cite{C-M-T-T15}, \cite{BHJOT}). Neocortical and hippocampal networks are composed of a mixture of excitatory and inhibitory neurons, and excitatory and inhibitory synaptic coupling can have different influences on the synchronization of neural firing. The presence of excitation in such networks affects the presence and characteristics of synchronized oscillations (see \cite{A-E-V94}, \cite{Br2000}). Recurrent inhibition plays an important role in the generation of synchronized oscillations in different systems (see \cite{Br-Ha99}, \cite{Ma-La96}, \cite{Je-Tr-Wh95}). It has been shown that spiking activity is unstable if all neurons are excitatory, while spontaneous activity becomes self-stabilizing in the presence of local inhibition (see \cite{Am-Br97}). In the cortex, the excitatory inputs are largely canceled out by the inhibitory ones (\cite{Se-Ts95}, \cite{Am-Br97}, \cite{Sh-Ne94}). The total synaptic input to a cell would be overwhelmingly depolarizing or hyperpolarizing unless the activity of the excitatory and inhibitory populations dynamically adjusts itself so that the large total inhibitory input nearly cancels the large excitatory one. If inhibitory neurons are included, spontaneous rates can be stabilized. Stability requires that the average local excitatory and inhibitory synaptic inputs to every neuron in the network closely balance out (\cite{Am-Br97}, \cite{So-Vr96}). In particular, in the cortex, considerable excitation is balanced by just enough inhibition to hold the population in check (see \cite{Lu-Om-Si06}, \cite{Abe91}).

What is important is that the balance between inhibition and excitation does not refer only to the local network but also to the global one, which means that the mathematical model should also allow for strong synapses between distant neurons (see \cite{So-Vr96}), a property that the PUS does not allow. Although the study of networks of excitatory and inhibitory neurons in realistic cortical conditions highlights the importance of a balance initially at a local level, it has been shown that what is crucial is to govern the stability of spontaneous activities also at the level of excitatory connections from outside the module (see \cite{Am-Br}). When the average synaptic efficacy exceeds a critical value, the local network develops a variety of local attractors in the background of the stability of the uniform global ones. Furthermore, at a local level, neurons whose rates are not significantly enhanced usually maintain a spiking rate at levels similar to the global spontaneous activity (\cite{Am-Br97}).

The paper is structured as follows. In Section \ref{LTB}, we first present two theorems that state conditions for the non-explosiveness of the model (Subsection \ref{s2.1}) and then a theorem about the uniqueness of the invariant measure (Subsection \ref{s2.2}).

\section{Long-time  behaviour }\label{LTB}
We start with the presentation of the model (see for instance \cite{H-K-L}). Let   $(N^i(ds, dz))_{i\in \N}$ be a family of \textit{i.i.d.}\ Poisson random measures on $\R_+ \times \R_+ $ having intensity measure $ds dz.$ We study the Markov process $X_t = (X^{ i }_t )_{i\in \N}$
taking values in $\R_+^{\N}$  and solving, for $i\in \N$, for $t\geq 0$,
\begin{eqnarray}\label{eq:dyn}
X^{ i}_t &= & X^{i}_0 -a_i\int_0^{t}g(X^{i}_s)ds -   \int_0^\infty  \int_0^t
X^{ i}_{s-}\I  _{ \{ z \le  \phi_{i} ( X^{ i}_{s-}) \}} N^i (ds, dz) \\
&&+    \sum_{ j \neq i } W_{j  i} \int_0^\infty \int_0^t  \I_{ \{ z \le \phi_{i} ( X^{j}_{s-}) \}} N^j (ds, dz).
\nonumber
\end{eqnarray}  where the intensity  functions $\phi_{i}:\R_+\mapsto \R_+$  are Lipschitz functions and  $g:\R_+\rightarrow \R_+$ is an integrable function, while for each $j\neq i,\, W_{j  i} \in \R$ is the synaptic weight describing the influence of neuron $j$ on neuron $i$.

For any test function $ f : \R_+^N \to \R $  and $x \in \R_+^{\N}$, the  generator of the process $X$  is described by    
\begin{align}\label{gen1}
\Ge f (x ) =&-\sum_{i\in \N}a_ig(x^i)\frac{d}{dx^i}f(x) +\sum_{i\in \N }   \phi_i (x^i) \left[ f ( \Delta_i ( x)  ) - f(x) \right]  
\end{align}
for $a_i\geq 0$ for all $i\in \N$,
where
\begin{equation}\label{Delta1}
(\Delta_i (x))_j =   
\begin{cases}
\max\{x^{j} +W_{i  j},0\}  & j \neq i \\
0 & j = i 
\end{cases}
.\end{equation}
 We consider both positive and negative synaptic weights $W_{ij}$. In the case of  non negative weights $W_{ij}\geq0$ we will write  $w_{ij}:=W_{ij}$, while in the case of negative weights $W_{ij}<0$ we denote  $v_{ij}:=\vert W_{ij}\vert$. The  sets of positive and negative synaptic weights that a neuron $i$ sends to other  neurons    are   denoted as $P_i$ and $N_i$ respectively, that is
 \begin{equation*}\label{eq:delta3}
 W_{i  j}=    
\begin{cases}
-v_{ij}  & j \in N_i \\
w_{ij}  & j \in P_i   
\end{cases}
 ,
\end{equation*}
with  $v_{ij}, w_{ij}\geq 0$ and sets $N_i \cap P_i=\emptyset$ such that $N_i\cup P_i=\N\setminus \{i\}$ for every $i \in \N$.

\subsection{Non-explosiveness}\label{s2.1}
 In order to examine criteria for the non-explosiveness of the system, we consider  conditions about the intensity functions and the weights.

  For $N[s,t]$ the number of spikes occurring during the time interval $[s,t]$ in the system, we say that the system is non-explosive if $\E(N[s,t])< \infty $ (see \cite{D-L-O}).

 Concerning the intensity functions $\phi_i$ the main assumption follows:

 \begin{itemize}

\item   \underline{Condition (L):} Assume that for every $i\in \N$,  $\phi_{i}$ are Lipschitz continuous with Lipschitz
constant $\vert \vert \phi_i\vert \vert_{Lip}$ and that they satisfy:
\[    \sum_{i=1}^{\infty}\vert \vert \phi_i\vert \vert_{Lip}<\infty, \   \  \  \  \ \  \ \  \sum_{i=1}^{\infty}\phi_i(0)<\infty  \  \  \  \ \  \ \text{and}\  \  \  \ \ \sup_{ j \in \N}\sum_{i\in N_{j}}v_{ji}\vert \vert \phi_i\vert \vert_{Lip}<\infty.\]
\end{itemize}

Since  any neuron $i$ spikes with intensity $\phi_i$, the first two conditions above mean that not only the neurons do not spike with the same frequency, but that this diminishes sufficiently fast from one to the other, so that the two first sums will be finite. We will consider the following labeling of neurons. Considering the intensity functions, we label neurons in such an order that  $\vert \vert \phi_{i+1}\vert \vert_{Lip}\leq \vert \vert \phi_i\vert \vert_{Lip}$.  Then, the possibility of a neuron $i$ to spike decreases  as $i$ increases.    Concerning the third bound, for neurons with this labeling order,   if the negative weights   that a neuron $j$ sends when it spikes   do not increase in  order, i.e.  $v_{ji}\geq v_{ji+1}$, then when  $\max_jv_{j1}<\infty$ the third bound follows from the first. In the opposite case, where for any fixed $i$ the $v_{ij}$ increases with $j$, the third inequality can still hold true if the decrease  of $\vert \vert \phi_i\vert \vert_{Lip}$ is big enough to compensate for the increasing weights.  

   Concerning the weights,  we consider four  different sets of conditions:

 \begin{itemize}

\item   \underline{Condition (A):}  The negative weights dominate over the positive ones. For all $i$:
\[ \sum_{j\in N_i} v_{ij}   \geq  \sum_{j\in P_i} w_{ij}. \]
 
 \item    \underline{Condition (B):} The drift dominates over the weights: 

Assume $\phi_{i}\geq 0$ is Lipschitz  continuous, with  Lipschitz  constant  $\vert \vert \phi_i\vert \vert_{Lip}$. Assume $\phi_{i} \leq c_{i}g$ for some $c_{i}>0$ and a non-negative function $g$, and that the weights, for all $i$, satisfy
 \[ \sum_{j\in N_i}v_{ij} <  \sum_{j\in P_i}w_{ij}\text{ \ and \  \  \ }a_i\frac{ \vert \vert \phi_i\vert \vert_{Lip}}{c_i}\geq   \sum_{j\in P_i}w_{ij}\vert \vert \phi_j\vert \vert_{Lip}- \sum_{j\in N_i}v_{ij} \vert \vert \phi_j\vert \vert_{Lip}.\]

 \item   \underline{Condition (C):}   Controlled dominance  of positive weights over the negative ones:
 \[\text{  for all \ } i: \  \  \sum_{j\in N_i}v_{ij} <  \sum_{j\in P_i}w_{ij}\]
 and  \[\max_{i\in \N} \left( \sum_{j\in P_i}w_{ij}-\sum_{j\in N_i} v_{ij}   \right)<\infty.\]

 \item    \underline{Condition (D):}  Interactions of infinitely long distance, with the negative weights dominating over the positive:
  For all $i$ and for any $k \in \N$
\[ \sum_{j\in N_i\cap \{j :\vert j-i\vert\leq k \}}v_{ij}  \geq  \sum_{j\in P_i\cap \{j :\vert j-i\vert\leq k \}}w_{ij} .\]
 \end{itemize}
 Since the purpose of the paper is to present criteria for the existence of the G-L model for the case that the synaptic weights $W_{ij}$ do not satisfy the usual PUS condition (\ref{Dob}), before we present the statement and the proof of the first theorem, we will present  examples  that satisfy    conditions (A), (B), (C) or (D), but  go beyond the   (\ref{Dob}), which means that they satisfy  the following:
 \begin{align}\label{NDob}\sup_{i\in \N}\sum_{j\in \N} \vert W_{ji} \vert=\infty.\end{align}
 \begin{example}
\label{par1} The first two examples relate to Condition (A). Notice that in any of these two first examples, since  one can consider any  $a_i\geq 0$, we can also set $a_i=0$ and get rid of the drift.
 
 \underline{Example 1:} We present an example of synaptic weights $W_{ij}$ that satisfy condition (A) but not (\ref{Dob}). Choose  the synaptic  weights as follows: For every $i\in \N$,
 \begin{equation*}W_{ij}=     \begin{cases}
-i  & j= i+1 \\
i  & j= i+2    \\
0  & \text{otherwise}   \\
\end{cases}  .
\end{equation*} Then, it is easy to see that 
\[ \sum_{j\in N_i}v_{ij}=   \sum_{j\in P_i}w_{ij}=i\]
while for every $i\geq 3$ we have that 
\[\sum_{j\in \N} \vert W_{ji} \vert=v_{i-1,i}+w_{i-2,i}=(i-1)+(i-2)=2i-3\rightarrow \infty \text{\  as  \ } i\rightarrow \infty,\]
which shows (\ref{NDob}).

   \underline{Example 2:} Another paradigm   that satisfies condition (A), but this time with an increasing length of interactions. For every   $i\in \N$,  define the weights 
\[W_{ij}=    \begin{cases}{cc}
-i & j= i+1 \\
1 & j=i+2,...,2i   \\
0  & \text{otherwise}   \\
\end{cases} .\] 
Then,   we compute
\[ \sum_{j\in P_i} w_{ij} =i-1<i=\sum_{j\in N_i} v_{ij}\] which is Condition (A). Furthermore,    for every $i\geq 3$ we have that 
\[\sum_{j\in \N} \vert W_{ji} \vert=\sum_{j<i:w_{ji}\neq 0}w_{ji}+v_{i-1,i}\geq v_{i-1,i}=i-1\rightarrow \infty \text{\  as  \ } i\rightarrow \infty,\]
which shows (\ref{NDob}).

 \underline{Example 3:} In the  third example, we present  a paradigm   that satisfies condition (B). For any $p_i>0$ for $i\in\N$, consider $\phi_i(x)=p_ix,$ $g(x)=x$ and the drift $a_i=(i+1)p_{i+2}+ip_{i+1}, \forall i\in \N$. Then choose the synaptic  weights to be:
\begin{equation*}W_{ij}=    \begin{cases}
-i  & j= i+1 \\
i +1 & j= i+2    \\
0  & \text{otherwise}   \\
\end{cases}   .
\end{equation*}
From the choice of $\phi_i$ and $g$ we have that $\vert\vert\phi_i\vert\vert_{Lip}=p_i$ and $c_i=p_i$.  Then,  for every $i\in \N$ 
\[ \sum_{j\in P_i} w _{ij}=(i+1) >i =\sum_{j\in N_i} v_{ij}   \text{ \  and \ } \sum_{j\in P_i} w_{ij} -\sum_{j\in N_i} v_{ij} =1<2= a_i\] \[a_i\frac{ \vert \vert \phi_i\vert \vert_{Lip}}{c_i}= (i+1)c_{i+2}+ic_{i+1}\geq(i+1)c_{i+2}+ic_{i+1}=   \sum_{j\in P_i}w_{ij}\vert \vert \phi_j\vert \vert_{Lip}-\sum_{j\in N_i}v_{ij} \vert \vert \phi_j\vert \vert_{Lip}\]which is condition (B). At the same time,    for every $i\geq 3$ we have that 
\[\sum_{j\in \N} \vert W_{ji} \vert=v_{i-1,i}+w_{i-2,i}=(i-1)+(i-1)=2i-2\rightarrow \infty \text{\  as  \ } i\rightarrow \infty,\]
which shows (\ref{NDob}).

 \underline{Example 4:}  A paradigm for (C). Choose the synaptic weights 
\begin{equation*}w_{ij}=    \begin{cases}
i+1  & j=i-1 \\
0 & \text{otherwise}  \\
\end{cases}   \text{  \  and  \  }  v_{ij}=     \begin{cases}
i  & j=i+1  \\
0   & \text{otherwise}   \\
\end{cases}  .
\end{equation*}
Then  \[ 0<\sum_{j\in P_i} w_{ij} -\sum_{j\in N_i} v_{ij} = 1 , \] and so,    \[\max_{i\in \N} \left( \sum_{j\in P_i}w_{ij}-\sum_{j\in N_i} v_{ij}   \right)=1<\infty,\]
 while \[  \sum_{j\in \N} \vert W_{ji}\vert =  w_{i+1i}+  v_{i-1i}=(i+1+1)+(i-1)=2i+1\rightarrow \infty\]
 as $i\rightarrow \infty$, which violates the PUS.
 
 \underline{Example 5:} Here we  present  a paradigm   that satisfies Condition (D). As in the first two examples, the drift can also be set to zero. For every $i \in \N$, define the synaptic weights as follows:
\begin{equation*}W_{ij}=    \begin{cases}
-2 & j: dist(j,i)=1 \\
-1  &  j: dist(j,i)=2k+1,\forall k\in\N    \\
1  &  j: dist(j,i)=2k,\forall k\in\N   \\
\end{cases}  .
\end{equation*}
Then,  for all $i\in \N$ and any $k$ odd,  we compute
\[ \sum_{j\in N_i\cap \{j :\vert j-i\vert\leq k \}}v_{ij} -\sum_{j\in P_i\cap \{j :\vert j-i\vert\leq k \}}w_{ij}=2>0,    \] while for every $k$ even 
\[ \sum_{j\in N_i\cap \{j :\vert j-i\vert\leq k \}}v_{ij} -\sum_{j\in P_i\cap \{j :\vert j-i\vert\leq k \}}w_{ij}=1>0,    \]
which is condition (D). Furthermore,   every neuron receives a weight from every other neuron, which has absolute value bigger or equal to one, and so  for every neuron $i$ we have
\[\sum_{j\in \N} \vert W_{ji} \vert= \infty,\]
which implies (\ref{NDob}).
\end{example}
  We now present the first result of the paper, a criterion for non explosion. What we will show is that during a finite time, only finite spikes can take place from the infinitely many neurons in the system.
 As a result, we can show the non explosiveness of the system, whenever it starts from any non-exploded configuration, i.e. any initial configuration $x_0$  such that $sup_{i\in \N}x^i_0<\infty,$ since the value of the  function 
 \begin{align}\label{LyapF}h(x_0)=\sum_{i=1}^{\infty}\vert \vert  \phi_i \vert\vert_{Lip}x_{0}^i\leq sup_{i\in \N}x^i_0\sum_{i=1}^{\infty}\vert \vert  \phi_i \vert\vert_{Lip}<\infty,\end{align}
at any non-exploded configuration is finite, because of Condition (L). 
\begin{theorem}\label{theorem1}Assume condition (L).  If any of the four  conditions (A), (B), (C) or (D) holds, then  the process is non-explosive.
\end{theorem}
\begin{proof}
Denote $b_{i}= (\vert \vert \phi_i\vert \vert_{Lip})^{-1}$ the inverse of the  Lipschitz constant of the intensity function $\phi_i$. Define $h(x)=\sum_{i=1} ^\infty \frac{x^i}{b_{i}} $ for $x^i\geq 0$. Then
\begin{align}\nonumber
\Ge h(x)=&-\sum_{i=1} ^\infty a_ig(x^i)\frac{d}{dx^i}h  +\\ \nonumber&+ \sum_{ i =1 }^{\infty}   \phi_{i} (x^i) \left[ \sum_{j\in N_i}\frac{\max\{ x^j-v_{ij},0\}}{b_{j}}+ \sum_{j\in P_i}\frac{(x^j+w_{ij}) }{b_{j}}+ \sum_{j\in( P_i\cup N_{i})^{c}}\frac{x^j  }{b_{j}}- \sum_{j=1}^{\infty}\frac{x^j}{b_{j}}\right] \\ \nonumber =&-\sum_{i=1} ^\infty a_ig(x^i)\frac{d}{dx^i}h  + \sum_{ i =1 }^{\infty}   \phi_{i} (x^i) \left[ -\sum_{j\in N_i:x^j>v_{ij}}\frac{   v_{ij} }{b_{j}} +\sum_{j\in P_i}\frac{ w_{ij} }{b_{j}} - \frac{x^i  }{b_{i}}-  \sum_{j\in N_{i}}\frac{x^j}{b_{j}}\mathcal{I}_{\{x^j\leq v_{ij}\}}\right]\\  \leq\ \nonumber & 
 -\sum_{ i =1 }^{\infty}  \frac{a_i }{b_{i}}g(x^i) +\sum_{i=1} ^\infty \phi_{i} (x^i) \left[ -\frac{ x^i}{b_{i}} -\sum_{j\in N_i} \frac{v_{ij}}{b_{j}}  + \sum_{j\in P_i} \frac{w_{ij}}{b_{j}}  \right]+\\ \nonumber  &+\sum_{i=1} ^\infty \phi_{i} (x^i) \left[   \sum_{j\in N_{i}}\frac{(v_{ij}-x^j)}{b_{j}}\mathcal{I}_{\{x^j\leq v_{ij}\}}  \right].\\ \nonumber \end{align}
If we use (L)
we get\begin{align}\label{t2.1_1}\nonumber
\Ge h(x)\leq&
-\sum_{ i =1 }^{\infty}  \frac{a_i }{b_{i}}g(x^i) +\sum_{i=1} ^\infty \phi_{i} (x^i) \left[ -\frac{ x^i}{b_{i}} -\sum_{j\in N_i} \frac{v_{ij}}{b_{j}}  + \sum_{j\in P_i} \frac{w_{ij}}{b_{j}}  \right]+\\  \nonumber  &+\max_{i}\left[   \sum_{j\in N_{i}}\frac{v_{ij} }{b_{j}}   \right] \sum_{i=1} ^\infty \phi_{i} (x^i) \\ \nonumber  \leqslant&
-\sum_{ i =1 }^{\infty}  \frac{a_i }{b_{i}}g(x^i) +\sum_{i=1} ^\infty \phi_{i} (x^i) \left[ -\frac{ x^i}{b_{i}} -\sum_{j\in N_i} \frac{v_{ij}}{b_{j}}  + \sum_{j\in P_i} \frac{w_{ij}}{b_{j}}  \right]+\\ &+\max_i\left[   \sum_{j\in N_{i}}\frac{v_{ij} }{b_{j}}   \right]h(x), 
\end{align}
where in the last inequality we used that $\phi_i$ is Lipschitz continuous and that $ \sum_{i=1} ^\infty \phi_{i} (x^i)\leq \sum_{i=1} ^\infty \vert \vert \phi_i\vert \vert_{Lip}x^i=h(x)$. 

  Under   hypothesis  (A), since $\phi_i, g\geq 0$, (\ref{t2.1_1})  becomes
\[\Ge h(x)\leq -\sum_{ i =1 }^{\infty}  \frac{a_i }{b_{i}}g(x^i) -\sum_{i=1} ^\infty \phi_{i} (x^i)\frac{ x^i}{b_{i}}   +\max_i\left[   \sum_{j\in N_{i}}\frac{v_{ij} }{b_{j}}   \right]h(x)\leq B_{1}h(x),\]
where above we denoted $B_1:= \max_i\left[   \sum_{j\in N_{i}}\frac{v_{ij} }{b_{j}}   \right].$

 Under   hypothesis  (B),   (\ref{t2.1_1}) becomes
\begin{align*}\Ge h(x)\leq &-\sum_{ i =1 }^{\infty}  \left(\frac{a_i }{b_{i}}g(x^i) +\phi_{i}(x^{i})\left[ \sum_{j\in N_i} \frac{v_{ij}}{b_{j}}  - \sum_{j\in P_i} \frac{w_{ij}}{b_{j}} \right] \right)-\sum_{i=1} ^\infty \phi_{i} (x^i)\frac{  x^i}{b_{j}} +\max_i\left[   \sum_{j\in N_{i}}\frac{v_{ij} }{b_{j}}   \right]h(x)  \\ \leq &-\sum_{ i =1 }^{\infty} g(x^i) \left(\frac{a_i }{b_{i}}+c_{i} \left[ \sum_{j\in N_i}  \frac{v_{ij}}{b_{j}}- \sum_{j\in P_i} \frac{w_{ij}}{b_{j}} \right] \right)-\sum_{i=1} ^\infty \phi_{i} (x^i)\frac{  x^i}{b_{i}}+\max_i\left[   \sum_{j\in N_{i}}\frac{v_{ij} }{b_{j}}   \right]h(x)  \\ \leq &B_{1}h(x).\end{align*}
where above at first we used that $\phi_{i}\leq c_{i}g$ and then the dominance of the weights by the drift. 

We now examine a bound under condition (C). At first we bound (\ref{t2.1_1}) by
\[
\Ge h(x)\leq   \sum_{i=1} ^\infty \phi_{i} (x^i) \left[ -\sum_{j\in N_i} \frac{v_{ij}  }{b_{j}}+ \sum_{j\in P_i} \frac{w_{ij}}{b_{j}}  \right]+\max_i\left[   \sum_{j\in N_{i}}\frac{v_{ij} }{b_{j}}   \right]h(x). \]
 Since $\phi_{i}$ is Lipschitz continuous  we can further bound   
\begin{align*}\Ge h(x)\leq & \sup_{i\in \N} \left( \sum_{j\in P_i} \frac{w_{ij} }{b_{j}}-\sum_{j\in N_i} \frac{v_{ij}}{b_{j}} \right)\sum_{i=1} ^\infty \vert\vert \phi_{i} \vert \vert _{Lip}x^i +  \sup_{i\in \N} \left( \sum_{j\in P_i} \frac{w_{ij} }{b_{j}}-\sum_{j\in N_i} \frac{v_{ij}}{b_{j}} \right)\sum_{i=1}^{\infty}\phi_{i}(0)+\\   &+\max_i\left[   \sum_{j\in N_{i}}\frac{v_{ij} }{b_{j}}   \right]h(x),\end{align*}
which leads to the following    condition
             \[\Ge h(x)\leq B_{2}h(x)+D_{2} \]
since $h(x)=\sum_{i=1} ^\infty\frac{x^i}{b_i}=\sum_{i=1} ^\infty \vert\vert \phi_{i} \vert \vert _{Lip}x^i$ and the two   constants below \[B_{2}= B_1+ \sup_{i\geq 1}\left( \sum_{j\in P_i}w_{ij}-\sum_{j\in N_i} v_{ij}   \right)  \ \text{and}  \   D_{2}= \sup_{i\in \N} \left( \sum_{j\in P_i} \frac{w_{ij} }{b_{j}}-\sum_{j\in N_i} \frac{v_{ij}}{b_{j}} \right)\sum_{i=1}^{\infty}\phi_{i}(0).\]
are finite because of condition (C) and (L). 

If we put together, (A), (B) and (C) we get 
             \[\Ge h(x)\leq B_{2}h(x)+D_{2.} \]
If we use  Dynkin's formula (see \cite{Dyn}, \cite{Oks}) :
\[\E[h(X_t)]=h(x_0)+\int_0^t\E[\Ge h(X_s)]ds,\]
we then get
\begin{align*}\E[h(X_t)]=h(x_0)+\int_0^t\E[\Ge h(X_s)]ds\leq     h(x_0)   + \int_0^t  D_2 + B_{2}\E[h(X_s)]ds.\end{align*}
We now consider  the linear generalisation of Gronwall's inequality (see \cite{Fil}): For continuous $u(t)>0$ and  $a,b\geq 0, c>0$
\begin{align*}u(t)\leq \alpha+\int_{t_0}^t b+c u(s)ds \Rightarrow    u(t)\leq \frac{b}{c}(e^{c(t-t_0)}-1)+ a e^{c(t-t_0)} , \ \ t\geq t_0.\end{align*}
Using the last one  we get
\[\E[h(X_t)]\leq h(x_0) e^{B_{2}t}+\frac{D_{2}}{B_{2}}(e^{B_2 t}-1).\]

For the case (D), one can consider the operator $\Ge_k$ to be the same as $\Ge$ but with the weights $X_{ij}=0$ if $\vert i-j \vert >k$. Then,  one works as before to obtain
 \[\E[h(\hat X^k_t)]\leq h(\hat X^k_0) e^{B_{2}t}+\frac{D_{2}}{B_{2}}(e^{B_2t}-1)\]
  where $\hat X^k$ is the process with generator $\Ge_k$. Then, we can use Fatou's lemma to extend the inequality to the infinite dimensional process
\[\E[h(X_t)]\leq \liminf_k\E[h(\hat X^k_t)]\leq h(x_0) e^{B_{2}t}+\frac{D_{2}}{B_{2}}(e^{B_2 t}-1).\]
Putting all the bounds from the  four cases together
\begin{align}\label{Gron}\E[h(X_t)]\leq h(x_0)+(h(x_0)  +\frac{D_{2}}{B_{2}})e^{B_2 t}.\end{align}
We will now show the non-explosiveness of the system. For $N[s,t]$ the number of   spikes occurring  during the time interval $[s,t]$ in the system, we compute
\begin{align*}\E(N[s,t])=&\E\int_s^t\sum_{i=1}^{\infty}\phi_i(x^i_u)du\leq \\  \leq &\E\int_s^t  \sum_{i=1}^{\infty}\vert\vert \phi_i \vert \vert _{Lip} x^i_u du+(t-s)\sum_{i=1}^{\infty} \phi_{i} (0)\\ =&\int_s^t\E[h(x_u)]du+(t-s)\sum_{i=1}^{\infty} \phi_{i} (0)  \end{align*}
where the second term on the right-hand side is bounded by hypothesis (L).
If we use (\ref{Gron})
\begin{align*}\E(N[s,t])\leq & \frac{(h(x_0)  +\frac{D_{2}}{B_{2}})}{B_{2}}(e^{B_2t}-e^{B_2 s})+( h(x_0)+\sum_{i=1}^{\infty} \phi_{i} (0))(t -s )<\infty,\end{align*}
which proves the theorem.
\end{proof}
In   Theorem \ref{theorem1}, we show the non-explosiveness of the system under strong conditions on the intensity functions, so that (\ref{LyapF}) will hold for any non-exploded configuration. Next, we will show that, under weaker conditions,     weaker results hold. We will show  that during a finite time, any single neuron can spike only a finite number of times. As a result, if any neuron is interacting only with a finite number of other neurons through synaptic connections,   then any single neuron can jump only finitely many times on a finite time interval.
  
Furthermore, under some additional conditions on the reset value of the intensity function  $\phi_i(0)$,
we show, that if the system starts from a configuration $x_0$ such that $\sum_{i=1}^{\infty}x_{0}^i<\infty,$
then on a finite time interval only finitely many neurons in the system will spike.  One should notice, that while under the strong condition (L) we could consider any non-exploded initial configuration,    now that (L) will not be assumed, we restrict to initial configurations $x_0$ such that $\sum_{i=1}^{\infty}x_{0}^i<\infty$.

We consider the following alternative to Condition (B).
 
\begin{itemize}
 
 \item    \underline{Condition (E):} The drift dominates over the weights: 
Assume $\phi_{i}$ is Lipschitz continuous and $\phi_{i} \leq c_{i}g$ for some $c_{i}>0,$ and that the weights satisfy
 \[ \sum_{j\in N_i}v_{ij} <  \sum_{j\in P_i}w_{ij}\text{ \ and \  \  \ }a_i\geq c_{i}   \sum_{j\in P_i}w_{ij} -c_{i} \sum_{j\in N_i}v_{ij}  .\] 
\end{itemize}
\begin{theorem}\label{theorem1+} Assume $\phi_{i}$ are Lipschitz continuous with   Lipschitz constants uniformly bounded and that   one of the four  conditions (A),   (C), (D) or (E) holds.
 Then the following statements are true:

(i) Any single neuron spikes at most finitely many times in a time interval of finite length. 

(ii) If  in addition
\[\sum_{i=1}^{\infty}\phi_i(0)<\infty,\]
then, in a time interval of finite length only finitely many  spikes take place in the system, assuming that the system's initial configuration $x_0$ is such  that $\sum_{i=1}^{\infty}x_{0}^i<\infty$.
\end{theorem}
\begin{proof}
Define $U(x)=\sum _{i\in K}x^i,$ for $x^{i}\geq 0$, where the set $K$ will be either  $K=\{i\}$ for some $i\in \N$, or $K=\N$, referring to the statement (i) and the statement (ii)  of the theorem respectively. We have
\begin{align}\label{t2.1_1a}
\Ge U(x)=  
-\sum_{ i\in K }a_i g(x^i) +\sum_{i \in K} \phi_{i} (x^i) \left[ - x^i -\sum_{j\in N_i} v_{ij}  + \sum_{j\in P_i} w_{ij}  \right]. 
\end{align}
Under   Hypothesis  (A), the last becomes
\[\Ge U(x)\leq -\sum_{ i\in K}  a_i g(x^i) -\sum_{i\in K} \phi_{i} (x^i)x^i   \leq 0, \]
while    under hypothesis  (E),   (\ref{t2.1_1a}) becomes
\begin{align*}\Ge U(x)\leq  -\sum_{ i \in K} g(x^i) \left(a_i  +c_{i} \left[ \sum_{j\in N_i} v_{ij}  - \sum_{j\in P_i} w_{ij} \right] \right)-\sum_{i=1} ^\infty \phi_{i} (x^i)  x^i   \leq 0.\end{align*}
Then if we work as in the proof of Theorem \ref{theorem1} for (A) and (D), we  obtain
\[\E[U(X_t)]\leq U(x_0).\] 

Under condition (C), we bound (\ref{t2.1_1a}) by
\[
\Ge U(x)\leq   \sum_{i=1} ^\infty \phi_{i} (x^i) \left[ -\sum_{j\in N_i} v_{ij}  + \sum_{j\in P_i} w_{ij}  \right], \]
 and using that   $\phi$ is Lipschitz   leads to the following  condition
\[\Ge U(x)\leq BU(x)+D \]
for the two constants $B=\sup_{i\in K}\left\{\vert\vert \phi_{i} \vert \vert _{Lip}\right\}\sup_{i\in K} \left\{ \sum_{j\in P_i} w_{ij} -\sum_{j\in N_i} v_{ij} \right\}$  and 

$D=\sum_{i  \in K}\phi_{i}(0) \sup_{i\in K} \left\{ \sum_{j\in P_i} w_{ij} -\sum_{j\in N_i} v_{ij} \right\}$.
If we  use  Dynkin's formula and Gronwall's inequality, as before, we will get
\[\E[U(X_t)]\leq U(x_0) e^{Bt}+\frac{D}{B}(e^{Bt}-1).\]

In a similar way as in Theorem \ref{theorem1}, for the case (D) we get: $\E[U(X_t)] \leq  U(x_0).$
Putting all the bounds from the  four cases together
leads to \begin{align}\label{Grona}\E[U(X_t)]\leq U(x_0)+(U(x_0) +\frac{D}{B})e^{Bt}.\end{align}
We will now show the non-explosiveness of the system. For $N_{i}(s,t)$ the number of spikes of a single neuron $i$ occurring  during the time interval $[s,t]$, we compute
\begin{align*}\E(N_{i}[s,t])=&\E\int_s^t \phi_i(x^i_u)du\leq  \vert\vert \phi_i \vert \vert _{Lip}\E\int_s^t   \E[U(x_u)]du +(t-s)  \phi_{i} (0). \end{align*}
 
If we use (\ref{Grona}) for $K=\{i\},$ we obtain \begin{align*}\E(N_{i}[s,t])\leq & \frac{(U(x_0) +\frac{D}{B})}{B}(e^t-e^s)+(t-s)\left(  \phi_{i} (0)+U(x_{0}) \right)<\infty.\end{align*}
which shows (i). For (ii), for $N[s,t]$ the number of spikes on the whole system we calculate
\[\E(N [s,t])=\E\int_s^t\sum_{i=1}^{\infty}\phi_i(x^i_u)du\leq \sup_{i\in \N}\vert\vert \phi_i \vert \vert _{Lip}  \int_s^t \E[ U(x_{u})] du+(t-s)\sum_{i=1}^{\infty} \phi_{i} (0).\]
If we use (\ref{Grona}) for $K=\N,$ we obtain
\begin{align*} \E(N [s,t]) \leq &\sup_{i\in \N}\vert\vert \phi_i \vert \vert _{Lip} \left((t-s)  U(x_{0})+(t^{2}-s^2)\frac{ U(x_0)B}{2}+\frac{D}{B}(e^t-e^s) \right)+  \\  & +(t-s)U(x_{0})<\infty,
\end{align*}which shows (ii).
\end{proof}
From the  statement (i) of the last theorem we obtain the following corollary.
\begin{cor}Assume $\phi$ is Lipschitz continuous and that   any of the four  conditions (A),   (C) (D) or (E) holds.
If     for any neuron $i\in \N$, the number of synaptic connections     $\#\{j:   \  W_{ji}\neq 0\}<\infty,$      then any single neuron can spike only finitely many times on a finite time interval. \end{cor}
It should be noticed that $\#\{j:   \  W_{ji}\neq 0\}<\infty$ is weaker than the PUS condition (\ref{Dob}), since it restricts the number of synapsis leading to a neuron and not the values of the synaptic weight. As a result, one can consider systems like the one presented in the first example of   Examples \ref{par1}, where $\#\{j:   \  W_{ji}\neq 0\}=2<\infty$, but $\vert W_{ij}\vert=i,$  and so the PUS condition is violated since $\sum_{j\in \N} \vert W_{ji} \vert=2i-3\rightarrow \infty$   as $i\rightarrow \infty.$

\subsection{Wasserstein contraction}  \label{s2.2}
We  want to determine conditions for the process to have at most one invariant measure. 

Having established conditions for the non-explosiveness of the process, we will investigate conditions that will guarantee the uniqueness of the invariant measure. To do this we will show the Wasserstein contraction for jump rates that are  Lipschitz continuous. Our approach follows  closely the paper of Duarte, L\"ocherbach and Ost    \cite{D-L-O}. 

For any $x\in \R^\N$ we consider   $\vert \vert x \vert\vert_1=\sum_{i=1}^{\infty}\frac{\vert x_i \vert}{2^{i}}  $.  We define as a coupling between two measures $\mu$ and $\nu$ on $\R^\N$ any probability measure on $\R^\N\otimes \R^\N$ whose marginals are $\mu$ and $\nu$. If we denote $\Gamma (\mu,\nu)$ the set of all couplings between $\mu$ and $\nu$, then define the Wasserstein distance between the two measures $\mu$ and $\nu$ by
\[W_1(\mu,\nu)=\inf \left\{ \int_{\R^\N}  \int_{\R^\N} \vert \vert x-y \vert\vert_1\gamma(dx,dy),\gamma \in \Gamma(\mu,\nu) \right\}.\]
\begin{theorem} \label{theorem2}
 Assume that each $\phi_i $ and  $g $ are Lipschitz continuous. Furthermore, assume   that 
\[  (W1a) \ \  \  \   \  \  \inf_{i\in \N} \left[ a_i - \left(\sum_{j\in N_i,j>i} v_{ij} +\sum_{j\in P_i,j>i} w_{ij} \right)\right]> 0\]
and
\[  (W1b) \ \  \  \   \  \  \inf_{i\in \N} \left[ \frac{a_i }{2^{s}}- \left(\sum_{j\in N_i,j<i} v_{ij} +\sum_{j\in P_i,j<i} w_{ij} \right)\right] >0\]
 where $s =\max\{r>0:W_{i,i-r}>0\}$. 

For every $y\geq x,$ assume: for each $  i\in \N$
\[ (W2) \ \  \  \   \  \ \phi_{i}(y)-\phi_{i}(x)\leq g(y)-g(x),\]
\[ (W3) \ \  \  \   \  \ k_1(y-x)\leq g(y)-g(x)\]
for some constant  $k_1>0$ and  $\forall i\in \N$
\[  (W4) \ \  \  \   \  \ x(\phi_{i}(y)-\phi_{i}(x)) \leq\phi_{i} (y)(y-x). \]
Then, for any choice of two probability measures $\mu$ and $\nu$ on $\B(\R^\N)$
\[W_1(\mu P_t,\nu P_t)\leq ke^{-dt}W_1(\mu,\nu).\]
As a result, under the conditions of Theorem \ref{theorem1} or \ref{theorem1+},  the process X has at most one invariant measure. \end{theorem}
\begin{proof}
We consider  the following coupling generator.
\begin{align}\label{L2_4}\nonumber 
\Ge_2H(x,y)=&-\sum_{i\in \N}a_ig(x^i)\frac{d}{dx^i}H(x,y) -\sum_{i}a_ig(y^i)\frac{d}{dy^i}H(x,y)\\  &  \nonumber+\sum_{i\in \N }  \phi_i (x^i)\wedge \phi_i (y^i) \left[ H ( \Delta_i ( x) , \Delta_i ( y)  ) - H(x,y) \right] \\  &
 \nonumber+\sum_{i\in \N }(  \phi_i (x^i)-\phi_i (y^i))_+ \left[ H ( \Delta_i ( x),y  ) -H(x,y) \right]
 \\  &  \nonumber+\sum_{i\in \N } (  \phi_i (y^i)-\phi_i (x^i))_+ \left[ H (x, \Delta_i ( y)  ) - H(x,y) \right]\\   =:&
 I_1+I_2+I_3+I_4
\end{align}
where $I_1$ above stands for the first two sums on the right hand side representing the drift, and $I_2$, $I_3$ and $I_4$ the remaining three sums. This is a coupling that makes the two processes jump together as much as possible. The $I_2$ term refers to the two processes jumping simultaneously, while $I_3$ and $I_4$ to only the process $X$ and $Y$ jumping correspondingly.

If we set  $H(x,y)=\vert \vert x-y \vert\vert_1$ we will compute the four terms $I_1, I_2,I_3$ and $I_4$. For the drift term, we have
\begin{align*}
I_1= \sum_{i\in \N}(-1)^{1+\I_{x^i<y^i}}a_ig(x^i)+ \sum_{i\in \N}(-1)^{1+\I_{y^i<x^i}}a_ig(y^i),
\end{align*}
while for the remaining three terms,  we respectively obtain
\begin{align*}
I_2 =&  -\sum_{i\in \N }  \phi_i (x^i)\I_{x^i<y^i} \left[ \frac{1}{2^i}(y^i-  x^i ) \right]    -\sum_{i\in \N }  \phi_i (y^i)\I_{y^i<x^i} \left[ \frac{1}{2^i}(x^i-  y^i ) \right],  
\end{align*}
\begin{align*}
I_3=&   \sum_{i\in \N }(  \phi _i(x^i)-\phi _i(y^i))\I_{x^i>y^i} \left[\sum_{j\in N_i} \frac{1}{2^j}\vert  x^j-v_{ij}-y^j \vert+  \sum_{j\in P_i} \frac{1}{2^j}\vert  x^j+w_{ij}-y^j \vert   \right] \\  &  +  \sum_{i\in \N }(  \phi _i(x^i)-\phi _i(y^i))\I_{x^i>y^i}\left[ \frac{1}{2^i}y^i -\frac{1}{2^i}( x^i-y^i )-\sum_{j\neq i,j\in N_i\cup P_i} \frac{1}{2^j}\vert  x^j-y^j \vert \right]
 \end{align*}
 and
 \begin{align*}
I_4=&   \sum_{i\in \N }(  \phi _i(y^i)-\phi _i(x^i))\I_{y^i>x^i}   \left[\sum_{j\in N_i} \frac{1}{2^j}\vert  y^j-v_{ij}-x^j \vert+  \sum_{j\in P_i} \frac{1}{2^j}\vert  y^j+w_{ij}-x^j \vert    \right] \\  &+   \sum_{i\in \N }(  \phi _i(y^i)-\phi _i(x^i))\I_{y^i>x^i} \left[  \frac{1}{2^i}x^i -\frac{1}{2^i}( y^i-x^i )-\sum_{j\neq i,j\in N_i\cup P_i} \frac{1}{2^j}\vert  y^j-x^j \vert \right].
 \end{align*}
In order to bound the last two terms that involve absolute values of sums and differences, we will use the following two simple calculations. At first, observe that  for $v_{ij},w_{ij}\geq 0$ we have
\[\vert  x^j-v_{ij}-y^j \vert-\vert  x^j -y^j \vert\leq v_{ij}\]
and
\[\vert  x^j+w_{ij}-y^j \vert-\vert  x^j -y^j \vert\leq w_{ij}.\]
 If we use these to bound the $I_3$ term, we get
\begin{align*}
I_3\leq   &  \sum_{i\in \N }(  \phi_i (x^i)-\phi_i (y^i))\I_{x^i>y^i}  \left[\sum_{j\in N_i} \frac{v_{ij}}{2^j}+\sum_{j\in P_i} \frac{w_{ij}}{2^j}+\frac{1}{2^i}( 2y^i-x^i ) \right],
 \end{align*}
 while for   $I_4$ we obtain
 \begin{align*}
I_4\leq    \sum_{i\in \N }(  \phi _i(y^i)-\phi_i (x^i))\I_{y^i>x^i}  \left[\sum_{j\in N_i} \frac{v_{ij}}{2^j}+\sum_{j\in P_i}\frac{w_{ij}}{2^j}+\frac{1}{2^i}( 2x^i-y^i ) \right].
 \end{align*}
 Gathering all the bounds for $I_1, I_2,I_3$ and $I_4$  together to bound (\ref{L2_4}), and after rearranging the terms,  we obtain      
 \begin{align*}
\Ge_2H(x,y)\leq &\sum_{i\in \N} \I_{x^i<y^i} \left(c_i(\phi_i(y^i)-\phi_i(x^i))-\frac{a_i}{2^i}(g(y^i)-g(x^i))\right)+ 
\\  & \sum_{i\in \N}\frac{ \I_{x^i<y^i}}{2^i}\left( x^i(\phi_i(y^i)-\phi_i(x^i)) -\phi_i (y^i)(y^i-x^i) \right)+
\\  &\sum_{i\in \N} \I_{y^i<x^i} \left(c_i(\phi_i(x^i)-\phi_i(y^i))-\frac{a_i}{2^i}(g(x^i)-g(y^i))\right)+ 
\\  & \sum_{i\in \N}\frac{ \I_{y^i<x^i}}{2^i}\left(y^i(\phi_i(x^i)-\phi_i(y^i)) -\phi_i (x^i)(x^i-y^i) \right)\\ =:& S_1+S_2+S_3+S_4,
 \end{align*}
 where above,  in order to simplify the exposition, we have  denoted $c_i=\sum_{j\in N_i} \frac{v_{ij}}{2^j}+\sum_{j\in P_i} \frac{w_{ij}}{2^j}$. 
 If we use Condition (W4) to bound the second and the fourth sum, we get 
 \[S_2+S_4\leq 0.\] 
 For the  first   sum we obtain  \begin{align*}S_1  \leq &   \sum _{i\in \N} \I_{x^i<y^i} \left(\frac{1}{2^{i}}( \sum_{j\in N_i,j>i} v_{ij}+\sum_{j\in P_i,j>i} w_{ij}-a_i)(g(y^i)-g(x^i))\right) \\ &+  \sum _{i\in \N} \I_{x^i<y^i} \left(( \sum_{j\in N_i,j<i} \frac{v_{ij}}{2^j}+\sum_{j\in P_i,j<i} \frac{w_{ij}}{2^j}-\frac{a_i}{2^i})(g(y^i)-g(x^i))\right) , \end{align*}
  where above we made use of condition (W2) and then distinguished on the interactions that precede and follow $i$.    Since for $j<i$ we have that $j\geq i-s,$ for  $s=\max\{r>0:W_{i,i-r}>0\}$, the last can be bounded by \begin{align*}S_1 &\leq  \sum _{i\in \N} \I_{x^i<y^i} \left(\frac{1}{2^i}( \sum_{j\in N_i,j>i} v_{ij}+\sum_{j\in P_i,j>i} v_{ij}-a_{i})(g(y^i)-g(x^i))\right) +\\  & +    \sum _{i\in \N} \I_{x^i<y^i} \left(\frac{1}{2^{i}}( \sum _{j\in N_i,j<i} 2^{s}v_{ij}+\sum_{j\in P_i,j<i}2^{s} w_{ij}-a_{i})(g(y^i)-g(x^i))\right).\end{align*}
  Next we can use (W1a) to bound the terms for $j>i$ and (W1b) the terms for  interactions from $j<i$   \begin{align*}S_1 & \leq      - \sum _{i\in \N} \I_{x^i<y^i} \left(\frac{1}{2^{i}}  (O_{1}+2^{s}O_{2})(g(y^i)-g(x^i))\right),\end{align*}
  where we have  denoted 
$O_1=  \inf_{i\in \N} \left[ a_i - \left(\sum_{j\in N_i,j>i} v_{ij} +\sum_{j\in P_i,j>i} w_{ij} \right)\right]$
and

$O_2=  \inf_{i\in \N} \left[ \frac{a _i}{2^s} - \left(\sum_{j\in N_i,j<i} v_{ij} +\sum_{j\in P_i,j<i} w_{ij} \right)\right]$. Next we can use (W3) to obtain
  \begin{align*}S_1& \leq      -k_{1} (O_{1}+2^{s}O_{2})\sum   _{i\in \N} \I_{x^i<y^i} \left(\frac{1}{2^{i}}  ( y^i - x^i)\right)\\ &=-k_{1} (O_{1}+2^{s}O_{2})H(x,y).\end{align*}
If we work similarly for the third sum, we will get 
 \[S_3 \leq -k_{1} (O_{1}+2^{s}O_{2})H(x,y).\] Gathering the   bounds together,\[\Ge_2 H(x,y) \leq -k_{1} (O_{1}+2^{s}O_{2}) H(x,y)\leq-d H(x,y)\]
 for the constant $d=k_{1} (O_{1}+O_{2})> 0$.  If we use  Dynkin's formula, we have 
\[\E_{x,y}H(X_t,\hat X_t)=H(x,y)+\int_0^t\E[\Ge_{2} H(X_s,\hat X_{s})]ds\leq H(x,y)-d \int_0^t\E[H(X_s,\hat X_{s})]ds,\]
which, from Gronwall's inequality (see \cite{Gron}), implies
 \[\E_{x,y}H(X_t,\hat X_t)\leq H(x,y)e^{-td},\] 
 or equivalently 
 \[\E_{x,y}\vert \vert X_t-\hat X_t\vert \vert_{1}\leq \vert \vert x-y\vert \vert _{1}e^{-td}.\] 
 This leads to
  \[W_1( P_t(x),P_t(y))\leq \vert \vert x-y\vert \vert _{1}e^{-td}.\] 
 We  consider now the optimal   $W_1-$coupling   of $\mu$ and $\nu$ and simulate the couple $(x,y)$ according to this. Call this coupling $\gamma$. If we integrate the left and the right hand side of the above inequality with respect to this coupling, we  obtain
 \[\int_{\R^\N}  \int_{\R^\N}W_1( P_t(x),P_t(y))\gamma(dx,dy)\leq e^{-dt} W_1(\mu,\nu) .\]  One notices that this coupling is not generally the optimal  coupling of $\mu P_t$ and $\nu P_t$ and so  
 \[W_1(\mu P_t,\nu P_t)\leq\int_{\R^\N}  \int_{\R^\N}W_1( P_t(x),P_t(y))\gamma(dx,dy),\]  which leads to   
 \[ W_1(\mu P_t,\nu P_t)\leq e^{-dt} W_1(\mu,\nu).\]
 This proves the first assertion of the theorem. For the uniqueness, one can see \cite{D-L-O}.
 \end{proof}
Below, we present   examples that satisfy the conditions of Theorem \ref{theorem2}.

\begin{example} One can immediately observe that the third example presented in Example \ref{par1} satisfies all the conditions of Theorem \ref{theorem2}.

Similarly,  if we set $\phi(x)=g(x)=x$, in any of the examples 1 or 2  presented in     Example  \ref{par1}, we can obtain further paradigms of neural systems  that satisfy the conditions of Theorem \ref{theorem2}, as long as we choose an appropriate drift $a_i$.  For both "example 1"  and  "example 2"  one can easily verify that  an increasing  drift $a_i=2i+1$ is sufficient.    

  \end{example}

 \[\]
 \textbf{Acknowledgements:}
 
 The author would like to thank Prof. Eva  L\"ocherbach.

\end{document}